\documentclass[12pt,a4paper]{article}
\usepackage{hyperref}
\usepackage{tikz}
\usepackage{amsmath}
\usepackage{amssymb}
\usepackage{amsthm}
\usepackage{framed}
\usepackage{multicol}
\usepackage[OT2,T1]{fontenc}
\usepackage{graphicx}
\usepackage[all]{xy}
\usepackage{color}
\usepackage{mathrsfs}
\usepackage{mathtools}
\usepackage{authblk}

\numberwithin{equation}{section}

\newcommand{\sm}{\wedge}

\newcommand{\im}{\mathrm{im}}

\newcommand{\Sp}{\mathrm{Sp}}

\newcommand{\F}{\mathbb{F}}
\newcommand{\Z}{\mathbb{Z}}

\newcommand{\GL}{\mathrm{GL}}

\newcommand{\Po}{\mathcal{P}}
\newcommand{\Br}{\mathbf{B}}
\newcommand{\St}{\mathrm{St}}

\newcommand{\Top}{\mathrm{Top}}

\usepackage[utf8]{inputenc}

\title{Equivariant Steinberg Summands}
\author{Krishanu Sankar}
\date{}

\begin{document}

\maketitle

\begin{abstract}
We construct Steinberg summands of $G$-equivariant spectra with $\mathrm{GL}_n(\mathbb{F}_p)$-action. We prove a lemma about their fixed points when $G$ is a $p$-group, and then use this lemma to compute the fixed points of the Steinberg summand of the equivariant classifying space of $(\mathbb{Z}/p)^n$. These results will be used in a companion paper to study the layers in the mod $p$ symmetric power filtration for $H\underline{\mathbb{F}}_p$.
\end{abstract}
\newtheorem{theorem}{Theorem}
\newtheorem{lemma}[theorem]{Lemma}
\newtheorem{definition}[theorem]{Definition}
\newtheorem{proposition}[theorem]{Proposition}
\newtheorem{corollary}[theorem]{Corollary}
\newtheorem{conjecture}[theorem]{Conjecture}

\theoremstyle{definition}
\newtheorem{mydef}[theorem]{Definition}
\newtheorem{question}[theorem]{Question}
\newtheorem{example}[theorem]{Example}

\theoremstyle{theorem}
\newtheorem{exe}{Proposition}

\section{Introduction}
This brief paper establishes two results regarding Steinberg summands of equivariant spectra. Namely, let $G$ be a finite $p$-group and let $\GL_n=\GL_n(\F_p)$. Then,
\begin{enumerate}
\item (Theorem \ref{thm:commutingfunctors}) For any pointed $(G\times \GL_n)$-space $Y$, there is a natural homotopy equivalence $e_n(Y^G) \rightarrow (e_nY)^G$ from the Steinberg summand of the fixed points to the fixed points of the Steinberg summand.
\item (Theorem \ref{thm:HsummandofSteinberg}) Let $B_G(\Z/p)^n$ denote the equivariant classifying space of $(\Z/p)^n$. Let $\mathcal{C}$ denote the set of normal subgroups $H\subseteq G$ such that $G/H$ is an elementary abelian $p$-group. Then the fixed points of the Steinberg summand $e_n(B_G(\Z/p)_+)$ decomposes into a wedge sum of spectra
$$(e_n(B_G(\Z/p)_+^n)^G\simeq \bigvee\limits_{H\in \mathcal{C}}E_n(H).$$
If $G/H$ is elementary abelian of rank $d$, then the summand $E_n(H)$ is
$$E_n(H)\simeq e_{n-d}B(\Z/p)_+^{n-d}\sm \Sigma^{1-d}\Br_d^{\Diamond}\sm B(\Z/p)_+^d.$$
\end{enumerate}

Along the way we gather certain results about (non-equivariant) Steinberg summands that are scattered in the literature. We also prove a result relating Steinberg summands and Stiefel varieties $V_d(\F_p^n)$ (Proposition \ref{lemma:stiefelvariety}), which is an important step in the proof of Proposition \ref{prop:HsummandofSteinberg}.

This paper is a companion to a larger work (\cite{SymmetricPowers}) in which the layers in the mod $p$ symmetric power filtration are calculated, with a view to understanding $H\underline{\F}_p\sm H\underline{\F}_p$. In the larger paper, we observe that the genuine $G$-spectrum $H\underline{\F}_p$ is the infinite mod $p$ symmetric power of the equivariant sphere spectrum $\Sigma^{\infty G}S^0$, and the layers of the filtration
$$\Sigma^{\infty G}S^0=\Sp_{\Z/p}^1(\Sigma^{\infty G}S^0)\subset \Sp_{\Z/p}^p(\Sigma^{\infty G}S^0)\subset \Sp_{\Z/p}^{p^2}(\Sigma^{\infty G}S^0)\subset \cdots \subset \Sp_{\Z/p}^{\infty}(\Sigma^{\infty G}S^0)=H\underline{\F}_p$$
are the $n$-fold suspensions of the Steinberg summands $e_nB_G(\Z/p)_+^n$.

\section{Steinberg Summands}
In this section, we construct Steinberg summands and prove some of their basic properties. The results of this section are not original work of the author, but they are scattered throughout the literature so we collect and prove the results important to us.

In subsection \ref{sec:defineSteinberg}, we define the Steinberg idempotent and Steinberg representation. In subsection \ref{sec:products} we define product maps relating these idempotents. The Steinberg summand in topology does not appear until subsection \ref{definition:steinbergsummand}. There, we give its definition (Definition \ref{def:Steinbergsummand}) in terms of the flag complex. Several properties of the flag complex are proven, which will later be useful to us.
\subsection{The Steinberg idempotent}
\label{sec:defineSteinberg}

The content of this subsection is well known --- the reader may refer to \cite{CR}, \cite{Ste}, or \cite{MP}. Fix a prime $p$, and let $\F_p$ denote the field of $p$ elements. Let $n$ be a positive integer, and write $\GL_n=\GL_n(\F_p)$ for brevity. Denote by $R_n$ the group algebra $\Z_{(p)}[\GL_n]$. Tensor products will be taken over $\Z_{(p)}$ unless otherwise specified.

Let $\Sigma_n \subset \GL_n$ be the subgroup of permutation matrices, and let $B_n \subset \GL_n$ be the Borel subgroup of upper triangular matrices. Associated to these two subgroups are elements $\overline{\Sigma}_n, \overline{B}_n$ in the group algebra $R_n$ defined by
$$\overline{\Sigma}_n:=\sum\limits_{\sigma \in \Sigma_n}(-1)^{\sigma}\sigma, \hskip 1in \overline{B}_n:= \sum\limits_{b\in B_n}b.$$
Lemma 2 of \cite{Ste} states that $\overline{\Sigma}_n\overline{B}_n\overline{\Sigma}_n\overline{B}_n=c_n\cdot\overline{\Sigma}_n\overline{B}_n$, where $c_n$ is the constant
$$c_n=\prod\limits_{i=1}^{n}(p^i-1).$$
The number $c_n$ is invertible in $\Z_{(p)}$ and therefore the element
$$e_n=\frac{1}{c_n}\cdot \overline{\Sigma}_n\overline{B}_n$$
is an idempotent in the group algebra $R_n$.
\begin{mydef}
The element $e_n=\frac{1}{c_n}\cdot \overline{\Sigma}_n\overline{B}_n$ is called the \emph{Steinberg idempotent.} For any left $R_n$--module $M$, there is the $\Z_{(p)}$--submodule
$$e_nM=\{ e_nm: m\in M\} \subset M$$
which is called the \emph{Steinberg summand} of $M$. Taking the Steinberg summand is a functor from left $R_n$--modules to $\Z_{(p)}$-modules.
\end{mydef}

Note that $R_n$ is both a left $R_n$--module and a right $R_n$--module. Therefore, $R_ne_n$ is a left $R_n$-submodule of $R_n$.

\begin{mydef}
The left $R_n$--module $R_ne_n$ is denoted by $\St_n$ and is called the \emph{Steinberg module}.
\end{mydef}
There is an isomorphism of $\Z_{(p)}$-modules
$$e_nM \cong (\St_n \otimes M)_{\GL_n}.$$
The Steinberg module has dimension $p^{\binom{n}{2}}$ over $\Z_{(p)}$ -- this fact is a direct corollary of Propositions \ref{prop:pnchoose2} and \ref{prop:flagcomplexproof}, proven in a later section.

%\begin{exe}
%Let $L_n$ denote the subgroup of lower triangular matrices with 1's along the diagonal. Then $\St_n$ is a one-dimensional free module over the ring $\Z_{(p)}[L_n]$. Therefore, $\mathrm{dim}_{\Z_{(p)}}(\St_n)=|L_n|=p^{\binom{n}{2}}$.
%\end{exe}
%\begin{proof}
%It suffices to show that the elements $\{Ae_n\}_{A\in L_n}$ form a $\Z_{(p)}$-basis for the free $\Z_{(p)}$-module $R_n$. For every matrix $M\in \GL_n$, define coefficients $c_{A,M}\in \Z_{(p)}$ such that
%$$Ae_n=\sum\limits_{M\in GL_n}c_{A,M}M.$$

%It is easily checked that for any matrices $A, A'\in L_n$, $\sigma\in \Sigma_n$, and $b\in B_n$,
%$$A'\cdot \sigma\cdot b=A \implies \sigma=b=\mathrm{Id} \; \text{and} \; A=A'.$$
%It immediately follows that if $A, A'$ are any two matrices in the group $L_n$, then
%$$c_{A',A}=\begin{cases}
%1 \hskip 0.2in A'=A\\
%0 \hskip 0.2in A'\neq A
%\end{cases}.$$
%Therefore, the elements $\{Ae_n\}_{A\in L_n}$ are linearly independent.
%*****

%\end{proof}

\begin{mydef}
The element $\hat{e}_n=\frac{1}{c_n}\cdot\overline{B}_n\overline{\Sigma}_n$ is called the \emph{conjugate Steinberg idempotent}. For any left $R_n$-module $M$, the $\Z_{(p)}$-submodule $\hat{e}_nM$ is called the \emph{conjugate Steinberg summand} of $M$.
\end{mydef}
The following two maps are inverse isomorphisms
$$e_nM \rightarrow \hat{e}_nM \hskip 1.5in \hat{e}_nM \rightarrow e_nM$$
$$\overline{\Sigma}_n\overline{B}_nm \mapsto \overline{B}_n\overline{\Sigma}_n\overline{B}_nm\hskip 1in \overline{B}_n\overline{\Sigma}_nm \mapsto \overline{\Sigma}_n\overline{B}_n\overline{\Sigma}_nm$$
because composing then in either order induces multiplication by the unit $c_n\in\Z_{(p)}^{\times}$. Therefore $e_nM$ and $\hat{e}_nM$ are isomorphic as $\Z_{(p)}$-modules.

\subsection{Products on Steinberg Summands}\label{sec:products}
Let $i$ and $j$ be positive integers. The block inclusion
$$\GL_i \times \GL_j \rightarrow \GL_{i+j}$$
$$(A, B) \mapsto \begin{pmatrix}A&0\\0&B \end{pmatrix}$$
gives a map $R_i\otimes R_j \rightarrow R_{i+j}$. We denote by $e_i\boxtimes e_j$ the image under this map of the idempotent $e_i\otimes e_j$ of $R_i\otimes R_j$. The idempotent $e_i\boxtimes e_j$ has the following relation to the idempotent $e_{i+j}$. Let $U_{i,j}$ denote the group of $(i+j)\times (i+j)$ matrices of the form $\begin{pmatrix}I_i & *\\ 0 & I_j \end{pmatrix}$. Let $\Sigma_{\mathrm{shuf}}(i,j)$ denote the set of $\binom{i+j}{i}$ permutations $\sigma$ with the property that
$$1\le a<b\le i\implies\sigma(a)<\sigma(b), \hskip 0.2in \text{and} \hskip 0.3in i+1\le a<b\le i+j\implies \sigma(a)<\sigma(b).$$

Such permutations are known as $(i,j)$--shuffle permutations. Define
$$\overline{U}_{i,j}=\sum\limits_{u\in U_{i,j}}u \hskip 1.5in \overline{\Sigma_{\mathrm{shuf}}(i,j)}=\sum\limits_{\sigma\in \Sigma_{\mathrm{shuf}}(i,j)}(-1)^{\sigma}\sigma$$
Then the following identities can be checked
$$\overline{U}_{i,j}\cdot\overline{B_i\times B_j}=\overline{B_i\times B_j}\cdot\overline{U}_{i,j}=\overline{B}_{i+j}$$
$$\overline{\Sigma_{\mathrm{shuf}}(i,j)}\cdot\overline{\Sigma_i\times \Sigma_j}=\overline{\Sigma_i\times \Sigma_j}\cdot\overline{\Sigma_{\mathrm{shuf}}(i,j)}=\overline{\Sigma}_{i+j}$$
$$\overline{U}_{i,j}\cdot \overline{\Sigma_i\times \Sigma_j}=\overline{\Sigma_i\times \Sigma_j}\cdot\overline{U}_{i,j}$$
Therefore,
\begin{equation}\label{eqn:Steinbergproductidentity}
\overline{\Sigma_{\mathrm{shuf}}(i,j)}\cdot\overline{U}_{i,j}(e_i\boxtimes e_j)=\frac{c_{i+j}}{c_ic_j}\cdot e_{i+j}
\end{equation}

\begin{mydef}\label{def:steinbergproduct}
There is a homomorphism of left $\Z_{(p)}[\GL_i\times \GL_j]$-modules
$$\St_i\otimes_{\Z_{(p)}}\St_j=\Z_{(p)}[\GL_i\times \GL_j](e_i\boxtimes e_j) \rightarrow \Z_{(p)}[\GL_{i+j}]e_{i+j}=\St_{i+j},$$
$$A(e_i\boxtimes e_j) \mapsto A\overline{\Sigma_{\mathrm{shuf}}(i,j)}\cdot\overline{U}_{i,j}(e_i\boxtimes e_j).$$
We refer to the map above as the \emph{Steinberg product.} The Steinberg product represents the following projection of summands which is functorial in the $R_{i+j}$-module $M$
$$(e_i\boxtimes e_j)M \rightarrow e_{i+j}M,$$
$$x \mapsto \overline{\Sigma_{\mathrm{shuf}}(i,j)}\cdot\overline{U}_{i,j}x.$$
\end{mydef}
The following two properties can be checked. \label{property:steinbergcommutativeassociative}
\begin{itemize}
\item (Associativity) The following diagram commutes.
$$\xymatrix{(e_i\boxtimes e_j\boxtimes e_k)M\ar[r]\ar[d] & (e_{i+j}\boxtimes e_k)M\ar[d]\\
(e_i\boxtimes e_{j+k})M\ar[r] & e_{i+j+k}M}$$
\item (Commutativity) Let $\sigma \in \Sigma_{i+j}$ be the shuffle permutation that increases every number by $j$ modulo $i+j$. Then $e_i\boxtimes e_j=\sigma^{-1}(e_j\boxtimes e_i)\sigma$, and therefore we have inverse isomorphisms $\sigma:(e_i\boxtimes e_j)M \rightarrow (e_j\boxtimes e_i)M$ and $\sigma^{-1}:(e_j\boxtimes e_i)M \rightarrow (e_i\boxtimes e_j)M$. The following diagram commutes.
$$\xymatrix{(e_i\boxtimes e_j)M\ar[r]^{\sigma}\ar[dr] & (e_j\boxtimes e_i)M\ar[r]^{\sigma^{-1}}\ar[d] & (e_i\boxtimes e_j)M\ar[dl]\\
& e_{i+j}M &}$$
\end{itemize}

For any fixed $n$, one should think of the various idempotents $\{e_{i_1}\boxtimes\cdots \boxtimes e_{i_k}\}_{i_1+\ldots+i_k=n}$ as functors from the category of left $R_n$--modules to the category of $\Z_{(p)}$--modules. The Steinberg product defines natural transformations among these functors, starting from the initial functor $e_1\boxtimes\cdots\boxtimes e_1$ and going to the final functor $e_n$.

\begin{proposition}
Let $f:e_{i+j}M \rightarrow (e_i\boxtimes e_j)M$ be the $\Z_{(p)}$-linear map 
$$f(e_{i+j}m) =\frac{c_ic_j}{c_{i+j}}\cdot (e_i\boxtimes e_j)e_{i+j}m.$$
Then the composition of $f$ with the Steinberg product as shown is the identity map.
$$\xymatrix{e_{i+j}M\ar[r]^{f} & (e_i\boxtimes e_j)M\ar[r] & e_{i+j}M}$$
\end{proposition}
\begin{proof}
This proposition is a direct result of Equation \ref{eqn:Steinbergproductidentity} and the fact that $e_{i+j}^2=e_{i+j}$.
\end{proof}

\subsection{The flag complex}\label{definition:steinbergsummand}

For any pointed space $X$, let $\Sigma X:= S^1\sm X$ denote the reduced suspension of $X$. Suppose that $X$ is a pointed topological space with $\GL_n$--action. In this section, we construct a spectrum $e_nX$ in a way that mirrors the algebra of the previous section. There is a splitting in the homotopy category of $p$-local spectra
$$\Sigma^{\infty}X \simeq e_nX \vee (1-e_n)X,$$
where $\vee$ denotes the wedge sum. The reason we must pass to spectra is because our construction involves desuspending spaces.

Fix a positive integer $n$, and let $\F_p^n$ denote a fixed $n$--dimensional vector space over the field $\F_p$. Let $\Br_n$ denote the nerve of the poset of subspaces of $\F_p^n$ that do not equal $0$ or $\F_p^n$. This poset, and therefore the associated nerve $\Br_n$, carries a left action of $\GL_n$. The following properties are well known, but are proved for the sake of completeness. All supporting proofs are deferred to the end of this section.

\begin{proposition}
The space $\Br_n$ has the homotopy type of a wedge of $p^{\binom{n}{2}}$ spheres of dimension $n-2$.
\end{proposition}
\begin{proposition}
There is an isomorphism of $\Z_{(p)}[\GL_n]$-modules
$$H_{n-2}(\Br_n;\Z_{(p)})\cong \St_n.$$
\end{proposition}

As constructed, $\Br_n$ is not a pointed space. Let $\Br_n^{\Diamond}$ denote the unreduced suspension of $\Br_n$. The space $\Br_n^{\Diamond}$ is the geometric realization of a simplicial set where the $k$-simplices are flags $[W_0\subseteq \cdots\subseteq W_k]$ with the property that either $W_0=0$ or $W_k=\F_p^n$ but not both.
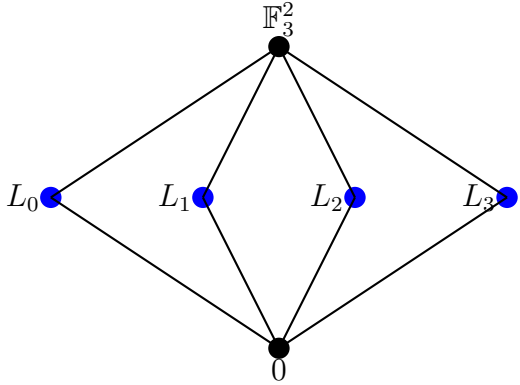
\begin{multicols}{2}
\begin{tikzpicture}

\fill (3,0) circle (4pt);
\fill [blue] (0,2) circle (4pt);
\fill [blue] (2,2) circle (4pt);
\fill [blue] (4,2) circle (4pt);
\fill [blue] (6,2) circle (4pt);
\fill (3,4) circle (4pt);
\draw [draw=black,thick] (3,0) -- (0,2);
\draw [draw=black,thick] (3,0) -- (2,2);
\draw [draw=black,thick] (3,0) -- (4,2);
\draw [draw=black,thick] (3,0) -- (6,2);
\draw [draw=black,thick] (3,4) -- (0,2);
\draw [draw=black,thick] (3,4) -- (2,2);
\draw [draw=black,thick] (3,4) -- (4,2);
\draw [draw=black,thick] (3,4) -- (6,2);
\node [below] at (3,0) {0};
\node [left] at (0,2) {$L_0$};
\node [left] at (2,2) {$L_1$};
\node [left] at (4,2) {$L_2$};
\node [left] at (6,2) {$L_3$};
\node [above] at (3,4) {$\F_3^2$};

\end{tikzpicture}

For example, let $p=3$. There are four one--dimensional subspaces of $\F_3^2$, which we denote by $L_0, L_1, L_2,$ and $L_3$. Pictured to the left is the topological space $\Br_2^{\Diamond}$. As a pointed space, it is homotopy equivalent to $\bigvee\limits_{3}S^1$. The blue points alone are $\Br_2$, which is homotopy equivalent to $\bigvee\limits_{3}S^0$.
\end{multicols}
\vskip 0.2in
Then $\Br_n^{\Diamond}$ is a pointed space with the 0-simplex $[0]$ as the basepoint. Its $\Z_{(p)}$-homology is as follows
$$\tilde{H}_*(\Br_n^{\Diamond}; \Z_{(p)})\simeq \begin{cases}
\St_n \hskip 0.3in *=n-1\\
0 \hskip 0.45in *\neq n-1
\end{cases}$$
If we smash the space $\Br_n^{\Diamond}$ by the negative sphere $S^{-(n-1)}$, we obtain the spectrum $\Sigma^{1-n}\Br_n^{\Diamond}$ whose homology is concentrated in degree 0. The spectrum $\Sigma^{1-n}\Br_n^{\Diamond}$ should be thought of as a topological analogue to the Steinberg module.
\begin{mydef}\label{def:Steinbergsummand}
Let $X$ be a spectrum with $\GL_n$-action. Then the {\bf Steinberg summand} of $X$, denoted $e_nX$, is defined as
$$e_nX = (\Sigma^{1-n}\Br_n^{\Diamond}\sm X)\sm_{\GL_n} (E\GL_n)_+$$
\end{mydef}
When $Y$ is any pointed space or spectrum with $\GL_n$-action, we henceforth use $Y_{h\GL_n}$ to denote the homotopy orbit space
$$Y_{h\GL_n}:= Y\sm_{\GL_n} (E\GL_n)_+.$$

\label{example:spectralsequence}As an example, let us compute the mod homology of $e_nX$, and show that it is equal to the Steinberg summand of the $R_n$-module $H_*(X;\Z_{(p)})$. The Hochschild--Serre spectral sequence associated to the fiber sequence
$$(\Sigma^{1-n}\Br_n^{\Diamond}\sm X) \rightarrow (\Sigma^{1-n}\Br_n^{\Diamond}\sm X)_{h\GL_n} \rightarrow B\GL_n$$
has $E^2$--page
$$E^2_{i,j}=H_i(B\GL_n; H_j(\Sigma^{1-n}\Br_n^{\Diamond}\sm X;\Z_{(p)}))\implies H_{i+j}((\Sigma^{1-n}\Br_n^{\Diamond}\sm X)_{h\GL_n};\Z_{(p)})$$
The homology group $H_0(\Sigma^{1-n}\Br_n^{\Diamond};\Z_{(p)})\cong \St_n$ is a projective $R_n$--module, and therefore flat. It follows by the K{\"u}nneth formula that
$$H_j(\Sigma^{1-n}\Br_n^{\Diamond}\sm X;\Z_{(p)})\simeq \St_n\otimes_{\Z_{(p)}} H_j(X;\Z_{(p)}).$$
Provided that $H_j(X;\Z_{(p)})$ is finite-dimensional over $\Z_{(p)}$, the $R_n$--module $\St_n\otimes_{\Z_{(p)}} H_j(X;\Z_{(p)})$ is projective, and so it has no higher $\GL_n$-homology. Thus, our $E^2$--page is
$$H_0(B\GL_n; \St_n\otimes_{\Z_{(p)}}H_*(X;\Z_{(p)})) = \St_n\otimes_{\Z_{(p)}[\GL_n]}H_*(X;\Z_{(p)}),$$
which is by definition the Steinberg summand $e_nH_*(X;\Z_{(p)})$. The $E^2$--page is concentrated on a single vertical line and therefore the spectral sequence collapses.

This argument is functorial in the pointed space $X$, and therefore implies the following commutative diagram of functors. Here, $\GL_n\Sp$ denotes the category of spectra of finite type with na\"{i}ve $\GL_n$-action, and $\mathrm{GrMod}_{R_n}$ denotes the category of graded left $R_n$--modules.
$$\xymatrix{\GL_n\Sp\ar[r]^{e_n(-)}\ar[d]_{H_*(-;\Z_{(p)})} & \Sp\ar[d]^{H_*(-;\Z_{(p)})}\\
\mathrm{GrMod}_{R_n}\ar[r]_{e_n(-)} & \mathrm{GrMod}_{\Z_{(p)}}}$$

As defined, we have no reason to believe the promise that $e_nX$ is a summand of $\Sigma^{\infty}X$. Proposition \ref{prop:stabletransfer} below implies that there are natural transformations of endofunctors on the category $\GL_n\Sp$,
$$e_n(-) \rightarrow \mathrm{Id}, \hskip 1in \mathrm{Id}\rightarrow e_n(-),$$
such that the composition $e_n(-) \rightarrow \mathrm{Id} \rightarrow e_n(-)$ is multiplication by a unit in $\Z_{(p)}$. Therefore, if we define the spectrum $(1-e_n)X$ as the homotopy fiber
$$(1-e_n)X := \mathrm{hofib}(\Sigma^{\infty}X\rightarrow e_nX)$$
then the cofiber sequence $(1-e_n)X \rightarrow \Sigma^{\infty}X \rightarrow e_nX$ splits, as promised.

\begin{proposition}\label{prop:stabletransfer}
There are maps
$$\Sigma^{1-n}\Br_n^{\Diamond} \rightarrow \Sigma_+^{\infty}\GL_n , \hskip 1in \Sigma_+^{\infty}\GL_n \rightarrow \Sigma^{1-n}\Br_n^{\Diamond}$$
such that the composition $\Sigma^{1-n}\Br_n^{\Diamond} \rightarrow \Sigma_+^{\infty}\GL_n \rightarrow \Sigma^{1-n}\Br_n^{\Diamond}$ is multiplication by a unit in $\Z_{(p)}$.
\end{proposition}

There are product maps as well. For any finite dimensional $\F_p$-vector space $V$, let $\Br_V$ denote the nerve of the poset of subspaces of $V$ which do not equal $0$ or $V$. If $\dim(V)=n$, then $\Br_V\simeq \Br_n$. Let $\tilde{\Br}_V$ denote the nerve of the poset of subspaces of $V$, including 0 and $V$ itself. The space $\tilde{\Br}_V$ is contractible, because the poset of subspaces of $V$ has an initial element 0. There is an obvious inclusion of simplicial sets $\Br_V^{\Diamond}\subset \tilde{\Br}_V$. If $V\simeq V'\oplus V''$, then there is a product map
$$\tilde{\Br}_{V'}\times \tilde{\Br}_{V''} \rightarrow \tilde{\Br}_V$$
$$(W', W'') \mapsto W'\oplus W''$$
When the above map is restricted to either $\tilde{\Br}_{V'}\times \Br_{V''}^{\Diamond}$ or $\Br_{V'}^{\Diamond}\times \tilde{\Br}_{V''}$, it lands in the subspace $\Br_V^{\Diamond}$. Therefore the product above restricts
$$(\tilde{\Br}_{V'}\times \Br_{V''}^{\Diamond}) \cup_{\Br_{V'}^{\Diamond}\times \Br_{V''}^{\Diamond}} \Br_{V'}^{\Diamond}\times \tilde{\Br}_{V''} \rightarrow \Br_V^{\Diamond}.$$
But since $\tilde{\Br}_{V'}$ and $\tilde{\Br}_{V''}$ are both contractible, the union above is homotopy equivalent to the unreduced join $\Br_{V'}^{\Diamond}\star \Br_{V''}^{\Diamond}$. We have constructed a product on flag complexes, namely
$$\Sigma \Br_{V'}^{\Diamond}\sm \Sigma\Br_{V''}^{\Diamond} \rightarrow \Sigma\Br_V^{\Diamond}.$$

If we choose isomorphisms $V'\cong \F_p^i$ and $V''\cong \F_p^j$, and the isomorphism $V' \oplus V'' \cong V$ is given by block inclusion, then it is easily checked when we take the top homology of the above product on flag complexes, we recover the product $\St_i \otimes \St_j \rightarrow \St_{i+j}$ on Steinberg modules.

\begin{proposition}\label{proposition:induction}
Let $V$ be a finite dimensional $\F_p$-vector space, and let $W$ be a subspace. Let $P_W\subset \GL(V)$ denote the parabolic subgroup of matrices preserving $W$. Let $\mathcal{S}_W$ denote the $P_W$-set
$$\mathcal{S}_W=\{W'\subset V: W+W'=V \; \text{and} \; W\cap W'=0\}.$$
Then the $P_W$-equivariant product map
$$\Sigma\Br_W^{\Diamond} \sm \bigvee\limits_{W'\in\mathcal{S}_W}\Sigma\Br_{W'}^{\Diamond} \rightarrow \Sigma\Br_V^{\Diamond}$$
is a homotopy equivalence of pointed spaces. The wedge sum shown is taken over all subspaces $W'$ complementary to $W$. (i.e. such that $W+W'=V$ and $W\cap W'=0$)
\end{proposition}

As promised, here are the proofs of Propositions 7 through 10.

\setcounter{theorem}{6}
\begin{proposition}\label{prop:pnchoose2}
The space $\Br_n$ has the homotopy type of a wedge of $p^{\binom{n}{2}}$ spheres of dimension $n-2$.
\end{proposition}

\emph{Note:} An alternate proof is given as (\cite{Ben}, Theorem 6.8.5).
\begin{proof}
We use induction on $n$. The case $n=1$ is obvious. Suppose that $n\ge 2$. Let $H\subset \F_p^n$ be a subspace of dimension $n-1$, and let $\Po\subset \Br_n$ be the nerve of the poset of subspaces which intersect $H$ nontrivially. The space $\Po$ is contractible, because it has a self map $W \mapsto H\cap W$ which is homotopic to both the constant map at $H$ and to the identity map. Therefore, $\Br_n\simeq \Br_n/\Po$. Note that any subspace of $\F_p^n$ of dimension 2 or greater automatically intersects $H$ nontrivially, and so the only simplices which remain in $\Br_n/\Po$ are those flags whose bottom space is a line transverse to $H$. Thus, $\Br_n/\Po$ decomposes as a wedge sum
$$\Br_n/\Po\simeq \bigvee\limits_{L\perp H}(\Br_n)_{\ge L}/(\Br_n)_{>L}$$
where $(\Br_n)_{\ge L}$ (resp. $(\Br_n)_{>L}$) denotes the nerve of the poset of subspaces containing $L$ (resp. strictly containing $L$). The space $(\Br_n)_{\ge L}$ is contractible, and $(\Br_n)_{>L}\simeq \Br_{n-1}$. Thus, $\Br_n/\Po\simeq \bigvee\limits_{L\perp H}\Sigma \Br_{n-1}$. The induction is now complete by the observation that there are $p^{n-1}$ lines transverse to $H$.
\end{proof}
\begin{proposition}\label{prop:flagcomplexproof}
The top homology group $H_{n-2}(\Br_n;\Z_{(p)})$, with its left $\GL_n$--action, is the Steinberg module $\St_n$.
\end{proposition}
\begin{proof}
The group of simplicial chains $C_{n-2}(\Br_n)$ is the free $\Z$-module over the set of maximal flags, which is $\Z[\GL_n/B_n]\cong \Z[\GL_n]\overline{B}_n$. Let $m\in\GL_n$ be a matrix, and let its columns be denoted by $v_1, \ldots, v_n$. Then consider the following linear combination of maximal flags
$$s_m=\sum\limits_{\sigma\in\Sigma_n}(-1)^{\sigma}(\langle v_{\sigma(1)}\rangle\subset \langle v_{\sigma(1)}, v_{\sigma(2)}\rangle\subset \cdots )$$
In the module $\Z[\GL_n]\overline{B}_n$, the chain $s_m$ is equal to $m\overline{\Sigma}_n\overline{B}_n$.

We claim that $s_m$ is a cycle, namely, $\partial(s_m)=0$. Let $w_1, \ldots, w_n$ be any permutation of $v_1, \ldots, v_n$. Then any $(n-3)$-simplex of the form $(\cdots\subset \langle w_1, \ldots, w_{i-1}\rangle \subset \langle w_1, \ldots, w_{i+1}\rangle\subset \cdots)$ is on the boundary of exactly the following two different $(k-2)$-simplices.
$$(\cdots\subset \langle w_1, \ldots, w_{i-1}\rangle \subset \langle w_1, \ldots, w_{i-1}, w_i\rangle \subset\langle w_1, \ldots, w_{i+1}\rangle\subset \cdots)$$
$$(\cdots\subset \langle w_1, \ldots, w_{i-1}\rangle \subset \langle w_1, \ldots, w_{i-1}, w_{i+1}\rangle \subset \langle w_1, \ldots, w_{i+1}\rangle\subset \cdots)$$
and this implies that $\partial(s_m)=0$.

We next claim that the set $\{s_m\}_{m\in\GL_n}$ spans $H_{n-2}(\Br_n)$. This will prove the claim that $H_{n-2}(\Br_n)\cong \St_n$. Fix a complete flag $\mathcal{F}=(\mathcal{F}_1\subset\cdots\subset \mathcal{F}_{n-1})$. Suppose that $(W_1\subset\cdots \subset W_{n-1})$ is a complete flag which is \emph{transverse} to $\mathcal{F}$, \emph{i.e.}, $W_i\cap \mathcal{F}_{n-i}=0$ for $i=1, \ldots, n-1$. For each $i=1, \ldots, n$, $W_i\cap \mathcal{F}_{n-i-1}$ is 1-dimensional, and so we may pick a sequence of nonzero vectors $w_1, \ldots, w_n$ so that $\langle w_i\rangle=W_i\cap \mathcal{F}_{n-i-1}$. The $w_i$'s have two important properties.
\begin{itemize}
\item Observe that $w_i\in \mathcal{F}_{n-i-1}$ and therefore $w_i\notin W_{i-1}$. Thus, by induction on $i$, $\langle w_1, \ldots, w_i\rangle=W_i$ for $i=1, 2, \ldots, n-1$, and $\langle w_1, \ldots, w_k\rangle=\F_p^n$.
\item Suppose that $\sigma\in\Sigma_n$ is a permutation such that $\sigma(i)=j$, where $j>i$. Then $w_j\in \mathcal{F}_{n-(j-1)}\implies w_j\in \mathcal{F}_{n-i}$, and because $w_{\sigma(i)}=w_j$, we have $\langle w_{\sigma(1)}, \ldots, w_{\sigma(i)}\rangle\cap \mathcal{F}_{n-i}\neq 0$. Thus, for any nontrivial permutation $\sigma\in\Sigma_n$, the flag $(\langle w_{\sigma(1)}\rangle\subset \langle w_{\sigma(1)}, w_{\sigma(2)}\rangle\subset\cdots)$ is not transverse to $\mathcal{F}$.
\end{itemize}
Let $m$ be the matrix whose columns are $w_1, \ldots, w_n$. The two properties above imply that
$$s_m=(W_1\subset \cdots \subset W_{n-1})+\sum\limits_{\sigma\neq \mathrm{id}}(-1)^{\sigma}(\mathrm{non}-\mathcal{F}-\mathrm{transverse}\;\mathrm{flags})$$
It follows that the dimension of the span of the set $\{s_m\}_{m\in\GL_kn}$ is at least as large as the number of complete flags transverse to $\mathcal{F}$, which is $p^{\binom{n}{2}}$. This is the dimension of the entire space $H_{n-2}(\Br_n)$, so the two are equal, as desired.
\end{proof}

%%%%%%%%%

\begin{proposition}\label{prop:stabletransfer}
There are maps
$$\Sigma^{1-n}\Br_n^{\Diamond} \rightarrow \Sigma_+^{\infty}\GL_n \hskip 1in \Sigma_+^{\infty}\GL_n \rightarrow \Sigma^{1-n}\Br_n^{\Diamond}$$
such that the composition $\Sigma^{1-n}\Br_n^{\Diamond} \rightarrow \Sigma_+^{\infty}\GL_n \rightarrow \Sigma^{1-n}\Br_n^{\Diamond}$ is multiplication by a unit in $\Z_{(p)}$.
\end{proposition}
\begin{proof}
For ease of notation, let us write $\Br=\Br_n^{\Diamond}$ in this proof. $\Br$ is a CW-complex whose $i$-cells correspond to flags of proper subspaces $(V_1\subsetneq \cdots \subsetneq V_i)$. $\Br$ has a skeletal filtration
$$\Br^{(0)} \subseteq \Br^{(1)} \subseteq \cdots \subseteq \Br^{(n-1)}=\Br$$
where $\Br^{(i)}$ contains the cells of dimension $i$ and lower.

Then $\Br^{(n-1)}/\Br^{(n-2)}$ is a wedge of $(n-1)$-spheres with a single sphere for every maximal flag. Thus,
$$\xymatrix{\Br=\Br^{(n-1)} \ar[r] & \Br^{(n-1)}/\Br^{(n-2)} \simeq (\GL_n/B_n)_+\sm S^{n-1}}$$
One may now compose with the stable transfer map $(\GL_n/B_n)_+\rightarrow (\GL_n)_+$ to obtain the map $\Br \rightarrow (\GL_n)_+\sm S^{n-1}$. On homology, this composite has the effect of sending $s_m\in \tilde{H}_{n-1}(\Br)$ to $m\overline{\Sigma}_n\overline{B}_n\in \tilde{H}_{n-1}((\GL_n)_+\sm S^{n-1})$

Now, we construct the map $(\GL_n)_+\sm S^{n-1}\rightarrow \Br$. Consider the composition
$$\xymatrix{(\GL_n)_+\sm S^{n-1}\ar[r]^{(-)\cdot \overline{\Sigma}_n} & (\GL_n)_+\sm S^{n-1}\ar[r] & (\GL_n/B_n)\sm S^{n-1}\simeq \Br^{(n-1)}/\Br^{(n-2)}}$$
For any matrix $A\in \GL_k$, the homology class of $(\GL_n)_+\sm S^{n-1}$ defined by $A$ maps under the above composition to the class $s_A$, which is a cycle. Therefore, there is a lift to $\Br^{(n-1)}$,
$$\xymatrix{& & \Br^{(n-1)}\ar[d]\\
(\GL_n)_+\sm S^{n-1}\ar@/_3ex/[drr]_{0}\ar@{-->}@/^3ex/[urr]\ar[rr]^{(-)\overline{\Sigma}_n\overline{B}_n} && \Br^{(n-1)}/\Br^{(n-2)}\ar[d]^{\partial}\\
& & \Sigma \Br^{(n-2)}}$$

The map above sends a matrix $A\in \tilde{H}_{n-1}((\GL_n)_+\sm S^{n-1})$ to $A\overline{\Sigma}_n\overline{B}_n\in \tilde{H}_{n-1}(\Br^{(n-1)}/\Br^{(n-2)})$, which lifts to the homology class $s_A\in\tilde{H}_{n-1}(\Br)$.

Therefore, the composite of the two maps we have constructed has the following effect in homology
$$\xymatrix{\tilde{H}_{n-1}(\Br)\ar[r] & \tilde{H}_{n-1}((\GL_n)_+\sm S^{n-1})\ar[r] & \tilde{H}_{n-1}(\Br^{(n-1)}/\Br^{(n-2)})}$$
$$s_A \mapsto A\overline{\Sigma}_n\overline{B}_n \mapsto A\overline{\Sigma}_n\overline{B}_n\overline{\Sigma}_n\overline{B}_n$$
Since $\overline{\Sigma}_n\overline{B}_n\overline{\Sigma}_n\overline{B}_n=c_n\overline{\Sigma}_n\overline{B}_n$ and $c_n\in\Z_{(p)}^{\times}$, the proposition has been proved.
\end{proof}

%%%%%%%%%%

\begin{proposition}
Let $V$ be a finite dimensional $\F_p$-vector space, and let $W$ be a subspace. Let $P_W\subset \GL(V)$ denote the parabolic subgroup of matrices preserving $W$. Let $\mathcal{S}_W$ denote the $P_W$-set
$$\mathcal{S}_W=\{W'\subset V: W+W'=V \; \text{and} \; W\cap W'=0\}.$$
Then the $P_W$-equivariant product map
$$\Sigma\Br_W^{\Diamond} \sm \bigvee\limits_{W'\in\mathcal{S}_W}\Sigma\Br_{W'}^{\Diamond} \rightarrow \Sigma\Br_V^{\Diamond}$$
is a homotopy equivalence of pointed spaces. The wedge sum shown is taken over all subspaces $W'$ complementary to $W$. (i.e. such that $W+W'=V$ and $W\cap W'=0$)
\end{proposition}
\begin{proof}
Let $V$ have dimension $n$. Both $\Sigma\Br_W^{\Diamond}\sm \bigvee\limits_{W'\perp W}\Sigma\Br_{W'}^{\Diamond}$ and $\Sigma\Br_V^{\Diamond}$ have underlying space equivalent to a wedge of copies of $S^n$, so it suffices to prove that the map is an equivalence on $n$-th homology groups. Without loss of generality, assume $V\simeq\F_p^n$, and $W\simeq \F_p^i$ is spanned by the first $i$ basis vectors. The $P_W$-set of subspaces $W'$ which are transverse to $W$ is equivalent to $P_W/(\GL_i\times \GL_{n-i})$. This set has size $p^{i(n-i)}$, and so by Proposition \ref{prop:pnchoose2},
$$\begin{aligned}
\mathrm{dim}(H_n(\Sigma \Br_W^{\Diamond}\sm \bigvee\limits_{W'\perp W}\Sigma\Br_{W'}^{\Diamond}))&=p^{\binom{i}{2}+i(n-i)+\binom{n-i}{2}}\\
&=p^{\binom{n}{2}}=\mathrm{dim}(H_n(\Sigma\Br_V^{\Diamond}))
\end{aligned}$$
So, by dimension reasons, it suffices to show that the given map is a surjection on homology. Recall that, for any $j$, the top $\Z_{(p)}$-homology group of $\Br_j^{\Diamond}$ is $\Z_{(p)}[\GL_j]\overline{\Sigma}_j\overline{B}_j$. Therefore, by the Kunneth formula,
$$\begin{aligned}
H_n(\Sigma \Br_W^{\Diamond}\sm \bigvee\limits_{W'\perp W}\Sigma\Br_{W'}^{\Diamond})&\simeq \mathrm{Ind}_{\GL_i\times \GL_{n-i}}^{P_W}\F_{(p)}[\GL_i\times \GL_{n-i}](\overline{\Sigma_i\times \Sigma_{n-i}})(\overline{B_i\times B_{n-i}})\\
&\simeq \Z_{(p)}[P_W](\overline{\Sigma_i\times \Sigma_{n-i}})(\overline{B_i\times B_{n-i}})
\end{aligned}$$
The map is given by the inclusion $P_W \rightarrow \GL_n$. Therefore, in order to show that the map
$$\xymatrix{\Z_{(p)}[P_W](\overline{\Sigma_i\times \Sigma_{n-i}})(\overline{B_i\times B_{n-i}})\ar[r] & \Z_{(p)}[\GL_n]\overline{\Sigma}_n\overline{B}_n}$$
is surjective, it is sufficient to show that any invertible $n\times n$ matrix can be written in the form $a \sigma b$, where $a, b\in B_n$ and $\sigma \in\Sigma_n$. This can be shown easily by row reduction.
\end{proof}

\section{Fixed points of a Steinberg summand}
\label{sec:commutingfunctors}
The Steinberg summand construction (Definition \ref{definition:steinbergsummand}) may be carried into the equivariant setting. First, we must define the $G$-equivariant analogue of homotopy orbits.
\begin{mydef}\label{def:eqclassifyingspace}
If $\Lambda$ is any finite group, then $E_G\Lambda$ denotes the $(G\times \Lambda)$-space whose fixed points under any subgroup $\Gamma\subset G\times \Lambda$ are
$$(E_G\Lambda)^{\Gamma}\simeq \begin{cases}
\star \; \mathrm{if} \; \Gamma\cap \Lambda = 1\\
\emptyset \; \mathrm{if} \; \Gamma\cap \Lambda \neq 1
\end{cases},$$
and $B_G\Lambda$ is the quotient $G$-space $(E_G\Lambda)/\Lambda$.
\end{mydef}
Note that the $G$-equivariant classifying spaces $B_G\Lambda$ fit into a theory of equivariant principal $\Lambda$-bundles (\cite{GMM}).
\begin{mydef}\label{def:eqSteinbergsummand}
Let $G$ be a finite group, and let $X$ be a spectrum with $(G\times \GL_n)$-action. The Steinberg summand $e_nX$ is the na{\"i}ve $G$-spectrum
$$e_nX = (\Sigma^{1-n}\Br_n^{\Diamond}\sm X)\sm_{\GL_n} (E_G\GL_n)_+.$$
\end{mydef}
It follows from Proposition \ref{prop:stabletransfer} that the na{\"i}ve $G$-spectrum $X$ contains $e_nX$ as a summand. Taking $G$-fixed points of the $(G\times \Lambda)$-space $E_G\Lambda$, yields the $\Lambda$-space $E\Lambda$. Thus we have an inclusion of $\Lambda$-spaces
$$E\Lambda \simeq (E_G\Lambda)^G \hookrightarrow E_G\Lambda.$$

This inclusion produces, for every subgroup $H\subseteq G$, a natural transformation from the composite functor $e_n((-)^H)$ to the composite functor $(e_n(-))^H$.
$$\xymatrix{\Top_*^{G\times \GL_n}\ar[r]^{(-)^H}\ar[d]_{e_n(-)} & \Top_*^{\GL_n}\ar[d]^{e_n(-)}\ar@{=>}[dl]\\
\text{Na{\"i}ve G-spectra}\ar[r]_{\hskip 0.2in (-)^H} & \text{Spectra}}$$
In this section, we prove that the natural transformation above is a homotopy equivalence when $G$ is a $p$-group. That is, we prove the following theorem.
\begin{theorem}\label{thm:commutingfunctors}
Let $G$ be a $p$-group, and let $H\subseteq G$ be any subgroup. Let $Y$ be any pointed $(G\times \GL_n)$-space. The inclusion of fixed points $E\GL_n\simeq (E_G\GL_n)^G \hookrightarrow E_G\GL_n$ induces a map
$$(\Br_n^{\Diamond}\sm Y^H)\sm_{\GL_n} (E\GL_n)_+ \rightarrow ((\Br_n^{\Diamond}\sm Y)\sm_{\GL_n} (E_G\GL_n)_+)^H.$$
If $G$ is a $p$-group, then the map above is an equivalence. It immediately follows that the map $e_n(Y^H) \rightarrow (e_nY)^H$ is an equivalence of spectra.
\end{theorem}
To prove this theorem, we must first establish a formula (Equation \ref{eq:fixedpointsformula}) for the fixed points of the equivariant homotopy orbits of a space.

\begin{mydef}
Let $G$ and $\Lambda$ be any finite groups, and let $H\subseteq G$ be a subgroup. For any homomorphism $f:H \rightarrow \Lambda$, its \emph{graph} is the subgroup of $H\times \Lambda$
$$\Gamma_f:=\{(h,f(h))\; :\; h\in H\}.$$
The group $\Lambda$ acts on the set $\mathrm{Hom}(H,\Lambda)$ by conjugation on the target. For a homomorphism $f:H \rightarrow \Lambda$, let $C_{\Lambda}(\im f)\subseteq \Lambda$ denote the centralizer of the image of $f$. Note that if $f, f'\in \mathrm{Hom}(H,\Lambda)$ are conjugate homomorphisms, then the centralizers $C_{\Lambda}(\im f)$ and $C_{\Lambda}(\im f')$ are conjugate subgroups.
\end{mydef}
Notice that if $f, f'\in \mathrm{Hom}(H, \Lambda)$ are two different homomorphisms, then the subgroup of $H\times \Lambda$ generated by $\langle \Gamma_f, \Gamma_{f'}\rangle$ is no longer a graph homomorphisms. It follows that $(E_G\Lambda)^{\Gamma_f}\cap (E_G\Lambda)^{\Gamma_{f'}}=\emptyset$.  For any $(G\times \Lambda)$-space $X$, we therefore obtain the following formula for the $H$-fixed points of $X\times_{\Lambda}E_G\Lambda$.
\begin{equation}\label{eq:fixedpointsformula}
\begin{aligned}(X\times_{\Lambda}E_G\Lambda)^H &= \left(\coprod\limits_{f\in \mathrm{Hom}(H,\Lambda)}X^{\Gamma_f}\times (E_G\Lambda)^{\Gamma_f}\right)/\Lambda\\
&= \left(\coprod\limits_{f\in \mathrm{Hom}(H,\Lambda)}X^{\Gamma_f}\right)\times_{\Lambda}E\Lambda\\
&\simeq \coprod\limits_{[f]\in \mathrm{Hom}(H,\Lambda)/\Lambda}(X^{\Gamma_f})_{hC_{\Lambda}(\im f)}.
\end{aligned}
\end{equation}
The map $E\Lambda=(E_G\Lambda)^G \hookrightarrow E_G\Lambda$ of $\Lambda$-spaces yields an inclusion map
$$X^H\times_{\Lambda}E\Lambda=(X\times_{\Lambda}E\Lambda)^H \hookrightarrow (X\times_{\Lambda}E_G\Lambda)^H.$$
Under the decomposition of Equation \ref{eq:fixedpointsformula}, the space $X^H\times_{\Lambda}E\Lambda$ is the summand corresponding to the zero homomorphism $H \rightarrow \Lambda$.

\begin{proof}[Proof of Theorem \ref{thm:commutingfunctors}]
The fixed point space $((\Br_n^{\Diamond}\sm Y)\sm_{\GL_n} (E_G\GL_n)_+)^H$ decomposes as a wedge sum
$$((\Br_n^{\Diamond}\sm Y)\sm_{\GL_n} (E_G\GL_n)_+)^H=\bigvee\limits_{\mathrm{Hom}(H,\GL_n)/\GL_n}((\Br_n^{\Diamond})^{\im f}\sm Y^{\Gamma_f})_{C_{\GL_n}(\im f)}.$$
We must prove that for every nontrivial homomorphism $f$, up to conjugacy, the summand $((\Br_n^{\Diamond})^{\im f}\sm Y^{\Gamma_f})_{C_{\GL_n}(\im f)}$ is contractible. It is sufficient to prove that the pointed $C_{\GL_n}(\im f)$-space $(\Br_n^{\Diamond})^{\im f}$ is equivariantly contractible. This will follow from a proof that the unpointed $C_{\GL_n}(\im f)$-space $(\Br_n)^{\im f}$ is equivariantly contractible, which follows from Lemma \ref{lemma:unipotence} below.

The map of spectra $e_n(Y^H) \rightarrow (e_nY)^H$ is the $(n-1)$-th desuspension of the inclusion $((\Br_n^{\Diamond}\sm Y)\sm_{\GL_n} (E\GL_n)_+)^H \hookrightarrow ((\Br_n^{\Diamond}\sm Y)\sm_{\GL_n} (E_G\GL_n)_+)^H$, and is therefore an equivalence.
\end{proof}
\begin{lemma}\label{lemma:unipotence}
Let $V$ be a finite dimensional vector space over a finite field $\F$ of positive characteristic $p$. Let $U\subset \GL(V)$ be a nontrivial unipotent subgroup (i.e. order a power of $p$). The fixed point space $(\Br_V)^U$ carries a residual action of the normalizer of $U$, which we denote by $N_{\GL(V)}(U)$. Then $(\Br_V)^U$ is $N_{\GL(V)}(U)$-equivariantly contractible.
\end{lemma}
\begin{proof}
The action of the group $U$ on the $\F$-vector space $V$ extends linearly to an action of the group ring $\F[U]$. Let $\mathcal{I}$ denote the \emph{augmentation ideal} of $\F[U]$, defined by generators
$$\mathcal{I}:=\langle u-1\rangle_{u\in U}.$$
Let $V'\subseteq V$ be the subspace annihilated by $\mathcal{I}$. Because $U$ contains at least one non-identity matrix, it must be that $V'\neq V$. The subspace $V'$ is preserved by the action of $N_{\GL(V)}(U)$. We claim that $V'\neq 0$. To prove this, it suffices to show that there is some $k>0$ such that $\mathcal{I}^k$ annihilates $V$. In the case where $U$ is a maximal unipotent subgroup of $\GL(V)$, the ideal $\mathcal{I}^{\dim(V)}$ annihilates $V$, and therefore for any unipotent subgroup $U$, $\mathcal{I}^k$ annihilates $V$ for some $k\le \dim(V)$.

Let $W\subsetneq V$ be a nonzero subspace that is preserved by $U$. Because $U$ is unipotent, $W$ has a vector $w$ such that $uw=w$ for every $u\in U$. This is equivalent to saying that $\mathcal{I}w=0$, and so it follows that the intersection $W \cap V'$ is nonzero. Thus, there is a well-defined $N_{\GL(V)}(U)$-equivariant poset map
$$f:(\Br_V)^U \rightarrow (\Br_V)^U$$
$$W \mapsto V'\cap W.$$
For every subgroup $\Gamma\subset N_{\GL(V)}(U)$, the map $f$ restricts to a map of fixed point spaces $f^{\Gamma}:((\Br_V)^U)^{\Gamma} \rightarrow ((\Br_V)^U)^{\Gamma}$. Because $V'\cap W \subset W$, the map $f^{\Gamma}$ is homotopic to the identity map. Because $V'\cap W \subset V'$, the map $f^{\Gamma}$ is homotopic to the constant map at $V'$. Therefore, the fixed point space $((\Br_V)^U)^{\Gamma}$ is contractible for every $\Gamma\subset N_{\GL(V)}(U)$, which completes the proof.

\end{proof}

\section{Fixed points in Equivariant Classifying Spaces}

Let $G$ be a $p$-group. The goal of this section is to compute the $G$-fixed points of the Steinberg summand of $B_G(\Z/p)_+^n$. To state the result of this computation, we must make a few definitions.

\begin{mydef}\label{definition:Frattini}
Let $\mathcal{C}$ denote the set of normal sugroups $H \unlhd G$ such that $G/H$ is an elementary abelian $p$-group. It is easily seen that the set $\mathcal{C}$ is closed under intersections, and thus $\mathcal{C}$ has a minimal element. As a poset, $\mathcal{C}$ is isomorphic to the poset of sub-$\F_p$-vector spaces of $G/F$. For each subgroup $H\in\mathcal{C}$, let $d(H)$ denote the rank of $G/H$ as an $\F_p$-vector space.
\end{mydef}

\begin{theorem}\label{thm:HsummandofSteinberg}
There is a decomposition of spectra
$$(e_nB_G(\Z/p)_+^n)^G \simeq \bigvee\limits_{H\in\mathcal{C}}E_n(H)$$
for spectra $E_n(H)$ which are given by the formula
$$E_n(H)\simeq e_{n-d(H)}B(\Z/p)_+^{n-d(H)}\sm \Sigma^{1-d(H)}\Br_{d(H)}^{\Diamond}\sm B(G/H)_+.$$
If $H$ and $K$ are subgroups such that $d(H\cap K)=d(H)+d(K)$, then the equivalence above respects the product structures on both sides.
\end{theorem}
The spectra $E_n(H)$ are defined in Definition \ref{def:H-summand}, along with their product structure. The equivalence $E_n(H)\simeq e_{n-d(H)}B(\Z/p)_+^{n-d(H)}\sm \Sigma^{1-d(H)}\Br_{d(H)}^{\Diamond}\sm B(G/H)_+$ is Proposition \ref{prop:HsummandofSteinberg}.

\subsection{The mod $p$ Stiefel variety}
\label{sec:stiefelvariety}
Let $n$ and $d$ be nonnegative integers. Let $V_d(\F_p^n)$ denote the set of $n\times d$ matrices with entries in the field $\F_p$, and with nullspace zero. Then $V_d(\F_p^n)$ is a finite set with an action of the group $\GL_n(\F_p)$. It is a mod $p$ analogue of the Stiefel manifold $V_d(\mathbb{R}^n)$ of orthonormal $d$-frames in Euclidean $n$-space. Note that there is an inclusion of $(\GL_m\times \GL_n)$-sets,
$$V_c(\F_p^m)\times V_d(\F_p^n)\hookrightarrow V_{c+d}(\F_p^{m+n})$$
given by block inclusion of matrices.

Let $\mathcal{F}$ be a functor from the category of finite dimensional mod $p$ vector spaces with isomorphisms, to the category $\Top_*$ of pointed spaces, such that
\begin{itemize}
\item For any finite dimensional mod $p$ vector spaces $V$ and $W$, there is an equivalence $\mathcal{F}(V\oplus W)\simeq \mathcal{F}(V)\sm \mathcal{F}(W)$ of $(\GL(V)\times \GL(W))$-spaces.
\item There is an equivalence $\mathcal{F}(0)\simeq S^0$.
\end{itemize}
For every integer $n\ge 0$, the pointed space $\mathcal{F}(\F_p^n)$ carries an action of the group $\GL_n(\F_p)$.
One may then consider its Steinberg summand $e_n\mathcal{F}(\F_p^n)=(\Sigma^{1-n}\Br_n^{\Diamond}\sm \mathcal{F}(\F_p^n))_{h\GL_n}$, which is a spectrum. These spectra are related by product maps
$$e_k\mathcal{F}(\F_p^k)\sm e_{\ell}\mathcal{F}(\F_p^{\ell}) \rightarrow e_{k+\ell}\mathcal{F}(\F_p^{k+\ell}),$$
which are built using the product $\Sigma^{1-k}\Br_k^{\Diamond}\sm \Sigma^{1-\ell}\Br_{\ell}^{\Diamond} \rightarrow \Sigma^{1-(k+\ell)}\Br_{k+\ell}^{\Diamond}$ and the block inclusion $\GL_k\times \GL_{\ell} \subset \GL_{k+\ell}$. In this section, we will prove the following lemma which relates the mod $p$ Stiefel variety $V_d(\F_p^n)$ to Steinberg summands.

\begin{lemma}\label{lemma:stiefelvariety}
Let $\mathcal{F}$ be a functor as above. Let $n, d$ be nonnegative integers such that $n\ge d$. Then there is an equivalence of spectra
$$(\Sigma^{1-d}\Br_d^{\Diamond}\sm \mathcal{F}(\F_p^d))\sm e_{n-d}\mathcal{F}(\F_p^{n-d})\rightarrow e_n(V_d(\F_p^n)_+\sm \mathcal{F}(\F_p^n)).$$
Denote the spectrum on the left by $A(n,d)$ and the spectrum on the right by $B(n,d)$. There are obvious product maps $A(n,d)\sm A(m,c) \rightarrow A(m+n,c+d)$ and $B(n,d)\sm B(m,c) \rightarrow B(m+n,c+d)$. Then the equivalence above respects these product maps.
\end{lemma}
\begin{proof}

Let $E_d \subset \F_p^n$ denote the subgroup spanned by the first $d$ coordinates, and let $E_{n-d}\subset \F_p^n$ denote the subgroup spanned by the last $n-d$ coordinates. Let $\Br_d^{\Diamond}:=\Br_{E_d}^{\Diamond}$ and $\Br_{n-d}^{\Diamond}:=\Br_{E_{n-d}}^{\Diamond}$. Let $\GL(E_d), \GL(E_{n-d}), \GL(\F_p^n,E_d)\subset \GL_n$ denote the subgroups of matrices
$$\GL(E_d)=\begin{pmatrix}
\GL_d & 0\\
0& I_{n-d}
\end{pmatrix}, \hskip 0.2in \GL(E_{n-d})=\begin{pmatrix} 
I_d & 0\\
0 & \GL_{n-d}
\end{pmatrix},\hskip 0.2in \GL(\F_p^n,E_d)=\begin{pmatrix}
I_d & *\\
0& \GL_{n-d}
\end{pmatrix}.$$
Then $\GL(\F_p^n,E_d)$ is the subgroup of matrices which act by the identity on $E_d$. Let $\mathcal{S}$ denote the set of subspaces $W\subset \F_p^n$ of dimension $(n-d)$ such that $W\perp E_d$. We have the following two observations:
\begin{enumerate}
\item As a $\GL(\F_p^n, E_d)$-torsor, $\mathcal{S}=\GL(\F_p^n, E_d)/\GL(E_{n-d})$.
\item As a $\GL_n$-torsor, $V_d(\F_p^n)=\GL_n/\GL(\F_p^n,E_d)$.
\end{enumerate}
Therefore,
$$\begin{aligned}
(\Sigma^{1-d}\Br_d^{\Diamond}\sm \mathcal{F}(\F_p^d))\sm e_{n-d}\mathcal{F}(\F_p^{n-d}) & \\
:= (\Sigma^{1-d}\Br_d^{\Diamond}\sm \mathcal{F}(E_d)\sm \Sigma^{1-n+d}\Br_{n-d}^{\Diamond}\sm \mathcal{F}(E_{n-d}))_{h\GL(E_{n-d})} & \hskip 0.2in \text{(by definition))}\\
\simeq (\Sigma^{1-d}\Br_d^{\Diamond}\sm \mathcal{F}(E_d)\sm \bigvee\limits_{W\in\mathcal{S}}\Sigma^{1-n+d}\Br_W^{\Diamond}\sm \mathcal{F}(W))_{h\GL(\F_p^n,E_d)} & \hskip 0.2in \text{(by (1) above)}\\
\stackrel{\simeq}{\rightarrow} (\Sigma^{1-n}\Br_n^{\Diamond}\sm \mathcal{F}(\F_p^n))_{h\GL(\F_p^n,E_d)} & \hskip 0.2in \text{(by Proposition \ref{proposition:induction})}\\
\simeq (\Sigma^{1-n}\Br_n^{\Diamond}\sm V_d(\F_p^n)_+\sm \mathcal{F}(\F_p^n))_{h\GL_n} & \hskip 0.2in \text{(by (2) above)}.
\end{aligned}$$
The fact that this equivalence respects the product maps is a routine check.
\end{proof}

\subsection{The fixed points of the Steinberg summand of an equivariant classifying space}\label{sec:fixedpoints}

Let $G$ be a finite $p$-group, and let $n$ be a positive integer. Any homomorphism from $G$ to $(\Z/p)^n$ has kernel contained in $\mathcal{C}$.
\begin{mydef}
For each subgroup $H\in \mathcal{C}$, let $$\mathrm{Hom}(G, (\Z/p)^n))[H]\subset \mathrm{Hom}(G, (\Z/p)^n)$$ denote the set of homomorphisms with kernel $H$.
\end{mydef}
Then $\mathrm{Hom}(G, (\Z/p)^n))=\bigsqcup\limits_{H\in\mathcal{C}}\mathrm{Hom}(G, (\Z/p)^n))[H]$. A homomorphism from $G$ to $(\Z/p)^n$ with kernel $H$ is the same as a monomorphism from $G/H$ to $(\Z/p)^n$. Thus, the $\GL_n$-torsor $\mathrm{Hom}(G, (\Z/p)^n))[H]$ is identified with the mod $p$ Stiefel variety $V_{d(H)}(\F_p^n)$ (see \ref{lemma:stiefelvariety}), and so
\begin{equation}\label{eqn:stiefeldecomposition}\mathrm{Hom}(G,(\Z/p)^n)=\bigsqcup\limits_{H\in\mathcal{C}}V_{d(H)}(\F_p^n).
\end{equation}

Now let us study the Steinberg summand of the $G$-fixed points of the equivariant classifying space $B_G(\Z/p)_+^n$. Equation \ref{eq:fixedpointsformula} tells us that
$$
\begin{aligned}
e_n((B_G(\Z/p)_+^n)^G)&\simeq e_n\left(\bigvee\limits_{\mathrm{Hom}(G,(\Z/p)^n)}B(\Z/p)^n_+\right)\\
&\simeq \bigvee\limits_{H\in\mathcal{C}}\left(e_n\bigvee\limits_{\mathrm{Hom}(G, (\Z/p)^n))[H]}B(\Z/p)_+^n\right).
\end{aligned}$$
\begin{mydef}\label{def:H-summand}
Let $n$ be a positive integer and $H\in\mathcal{C}$ be a subgroup of $G$. The spectrum
$$e_n\bigvee\limits_{\mathrm{Hom}(G, (\Z/p)^n))[H]}B(\Z/p)_+^n$$
is called the \emph{$H$-summand of $e_n((B_G(\Z/p)_+^n)^G)$.} We denote it by $E_n(H)$. For any two positive integers $m, n$ and subgroups $H, K\in\mathcal{C}$, there is a product map
$$E_m(H)\sm E_n(K) \rightarrow E_{m+n}(H\cap K)$$
which is determined by the product maps 
$$e_m((B_G(\Z/p)_+^m)^G)\sm e_n((B_G(\Z/p)_+^n)^G) \rightarrow e_{m+n}((B_G(\Z/p)_+^{m+n})^G).$$
\end{mydef}

Let $m, n$ be any two positive integers. There is an obvious isomorphism of $(\GL_m\times \GL_n)$-sets
$$\mathrm{Hom}(G,(\Z/p)^m)\times \mathrm{Hom}(G,(\Z/p)^n) \cong \mathrm{Hom}(G,(\Z/p)^{m+n}).$$
Under the identification of Equation \ref{eqn:stiefeldecomposition}, the isomorphism above yields product maps on the components
$$V_{d(H)}(\F_p^m)\times V_{d(K)}(\F_p^n) \rightarrow V_{d(H\cap K)}(\F_p^{m+n}).$$
In the situation where $d(H\cap K)=d(H)+d(K)$, this product is given by block inclusion of matrices, for an appropriate choice of basis. (Subgroups $H, K\subset G$ which satisfy the property that $d(H\cap K)=d(H)+d(K)$ are called \emph{transverse} in (\cite{SymmetricPowers}).)
\begin{proposition}\label{prop:HsummandofSteinberg}
There is an equivalence of spectra
$$E_n(H)\simeq e_{n-d(H)}B(\Z/p)_+^{n-d(H)}\sm \Sigma^{1-d(H)}\Br_{d(H)}^{\Diamond}\sm B(G/H)_+.$$
If $n,m$ are two positive integers and $H, K\in\mathcal{C}$ are transverse subgroups, then under the equivalence above, the product map $E_n(H)\sm E_m(K) \rightarrow E_{n+m}(H\cap K)$ is identified with the product on the right hand term that is built using the following three maps
$$e_{n-d(H)}B(\Z/p)_+^{n-d(H)}\sm e_{m-d(K)}B(\Z/p)_+^{m-d(K)} \rightarrow e_{m+n-d(H\cap K)}B(\Z/p)_+^{m+n-d(H\cap K)},$$
$$\Sigma^{1-d(H)}\Br_{d(H)}^{\Diamond}\sm \Sigma^{1-d(K)}\Br_{d(K)}^{\Diamond} \rightarrow \Sigma^{1-d(H\cap K)}\Br_{d(H\cap K)}^{\Diamond},$$
$$B(G/H)_+\sm B(G/K)_+ \cong B(G/(H\cap K))_+.$$
\end{proposition}
\begin{proof}
This is immediate from applying Lemma \ref{lemma:stiefelvariety} with the functor $\mathcal{F}((\Z/p)^n)=B(\Z/p)_+^n$.
\end{proof}

\end{document}